\documentclass[11pt]{amsart}
\usepackage{amssymb}
\usepackage{amsfonts}

\usepackage{bbm}
\usepackage{mathrsfs}
\usepackage{color}

\usepackage{array}

\usepackage{amsfonts,amsmath, amssymb}

\newcommand{\bea}{\begin{eqnarray}}

\newcommand{\eea}{\end{eqnarray}}

\newcommand{\be}{\begin {equation}}

\newcommand{\ee}{\end{equation}}

\newtheorem{theorem}{Theorem}[section]

\newtheorem{corollary}[theorem]{Corollary}
\newtheorem{lemma}[theorem]{Lemma}

\newtheorem{remark}[theorem]{Remark}

\newtheorem{claim}[theorem]{Claim}

\begin{document}

\title{prime divisors of sequences of integers}

\author {Xianzu Lin }

\date{ }
\maketitle
   {\small \it College of Mathematics and Computer Science, Fujian Normal University, }\\
    \   {\small \it Fuzhou, {\rm 350108}, China;}\\
      \              {\small \it Email: linxianzu@126.com}

\begin{abstract}

In this paper, we develop Furstenberg's topological proof of
infinity of primes, and prove several results about prime divisors
of sequences of integers, including the celebrated Schur's
theorem. In particular, we give a simple proof of a classical
result which says that a non-degenerate linear recurrence sequence
of integers of order $k>1$ has infinitely many prime divisors.
\end{abstract}






\section{Introduction}
Given a sequence of integers $\{a_n\}_{n=0}^{\infty}$, a prime $p$
is called a $prime\ divisor$ of  $\{a_n\}_{n=0}^{\infty}$ if
$p| a_n$ for some $n$.  This paper is mainly concerned with the
question  that when a  sequence of integers
$\{a_n\}_{n=0}^{\infty}$ has infinitely many prime divisors.

Euclid's theorem of infinity of primes says that the  sequence of
natural numbers has infinitely many prime divisors.

Euclid's proof of infinity of primes is very beautiful and simple,
which can be further applied to prove Schur's theorem
\cite{gb,sc}:

 \begin{theorem} \label{i}
Let $f(x)\in \mathbb{Z}[x]$ be a nonconstant integral polynomial.
Then the sequence $\{f(n)\}_{n=0}^{\infty}$ has infinitely many
prime divisors.
\end{theorem}

In 1955, Furstenberg gave a mystery proof of Euclid's theorem,
using the language of topology as follows:

For $a,b\in \mathbb{Z}$, $b>0$, let $Con(a,b)$ be the congruence
class $\{x\in \mathbb{Z}\mid x\equiv a (mod\ b)\}$. Then we obtain
a topology $ \mathcal{F}$ (Furstenberg's topology) by taking the
classes $Con(a,b)$ as a basis for the open sets.
 We note that each
$Con(a,b)$ is closed as well. If the set of primes $\mathbb{P}$
were finite, then
$$\mathbb{Z}\setminus \{1,-1\}=\bigcup_{p\in\mathbb{P}}Con(0,p)$$
would be also closed. Consequently $\{1,-1\}$ would be an open set,
but this contradicts the definition of the topology $\mathcal{F}$.

In \cite{ca,me}, it was shown that the topological language can be
avoided in Furstenberg's proof. In \cite[p.56]{da}, Furstenberg's
proof is treated as simply a reductio version of Euclid's. But, as
we will see in the following, Furstenberg's non-constructive proof
has special advantage in many cases; even the topological language
turns out to be very convenient!

Throughout this paper, $\mathbb{Z}$ denotes the set of
integers endowed with the Furstenberg topology $\mathcal{F}$, and $ \mathbb{N}$ and $ \mathbb{N}_0$  denote
the sets of positive and non-negative integers considered as topological
subspaces of $\mathbb{Z}$, respectively.

This work is supported by National Natural Science Foundation for
young (no.11401098). The author thanks the anonymous referee for
numerous suggestions and corrections about this paper.

 \section{Schur's theorem and its generalization}

In this section, we give a topological proof of Schur's theorem,
and derive a generalization. We first extend Furstenberg's proof
in the following form:

 \begin{theorem} \label{main1}
Let $f:\mathbb{N}\rightarrow \mathbb{Z}$ be a continuous map,
which is unbounded on any congruence class. Then the sequence
$\{f(n)\}_{n=1}^{\infty}$ has infinitely many prime divisors.
\end{theorem}

\begin{proof}
Choose an $m\in \mathbb{N}$ such that $|f(m)|=k>0$. Let $U$ denote
the open set $Con(0,k) = k\mathbb{Z}$; thus $U$ is an open
neighborhood of $k$. By the continuity of $f$, there is an open
neighborhood $V$ of $m$ (in $\mathbb{N}$) of the form $V:=
Con(m,b)_{|\mathbb{N}} = \{m+nb:n\in \mathbb{N}_0\}$ such that
$f(V )\subset U$, i.e., $f(m+nb)$ ($n=0,1,2\cdots$) is divisible
by $k$. Now, for each $n\in \mathbb{N}_0$, set
$g(n)=\tfrac{1}{k}f(m+nb)$. Then $g(n)$ is also a continuous map
 from $\mathbb{N}_0$ to $\mathbb{Z}$.  It suffices to show that
$\{g(n)\}_{n=0}^{\infty}$ has infinitely many prime divisors.
Assume that the set $\mathbb{P}_g$ of prime divisors of
$\{g(n)\}_{n=0}^{\infty}$  is finite. Then
$$M=\mathbb{N}_0\setminus \bigcup_{p\in\mathbb{P}_g}g^{-1}Con(0,p)$$ is an open subset of $\mathbb{N}_0$, consisting of those $n$ on which $g$ takes the value $\pm1$. But by the definition of $g$, we have $|g(0)|=1$. Hence $M$ is nonempty, and must contain a congruence class. This contradicts our assumption that $f$ is unbounded on any
congruence class.
\end{proof}

Since both addition and multiplication are continuous with respect
to Furstenberg's topology, any integral polynomial defines a
continuous map  from $\mathbb{N}$ to $\mathbb{Z}$. Hence Schur's
theorem follows directly from Theorem \ref{main1}.

Polynomial sequences belong to a very special kind of recurrence
sequences. In fact, an integral polynomial $f(x)$ of degree $k$
satisfies
\begin{equation}\label{for1}f(n+k+1)=\sum_{i=0}^k(-1)^i{k+1\choose
i+1}f(n+k-i).\end{equation} In the theorem below we generalize
Schur's theorem to a class of recurrence sequences.

 \begin{theorem} \label{w1}
Let $\{a_n\}_{n=0}^{\infty}$ be a recurrence sequence of integers
satisfying
$$a_{n+k+1}=\pm a_{n}+f(a_{n+1},\cdots a_{n+k})$$  where
$f\in\mathbb{Z}[x_{1},\cdots x_{k}]$. We further assume that $\lim
_{n\rightarrow\infty}|a_n|=\infty$. Then the recurrence
sequence has infinitely many prime divisors.
\end{theorem}
\begin{proof}
A routine argument about linear recurrence sequences
\cite[p.45]{ev1} shows that $\{a_n\}_{n=0}^{\infty}$ is periodic
modulo $m$, i.e., for any positive integer $m$, there exists
$s\in\mathbb{N}$ such that $a_n\equiv a_{n+s} (mod\ m)$ for all
$n\geq0$. Hence  a function $h:\mathbb{N}_0\rightarrow\mathbb{Z}$
of the form $h(n)=a_{n}$ is continuous. Now applying Theorem
\ref{main1} we conclude the proof.
\end{proof}

 Notice however that Theorem  \ref{w1} does not cover even the  linear recurrence
sequences satisfying the simple relation
$$a_{n+2}=a_{n+1}+2a_{n}.$$ In order to treat more recurrence
sequences, we need to further generalize Theorem \ref{main1}.

For a nonzero integer $m$, we define a variant of Furstenberg's
topology $ \mathcal{F}_m$ on $\mathbb{Z}$  by taking the classes
$Con(a,b)$ as a basis for the open sets, where $b$ runs over all
positive integers prime to $m$; then the symbol $ \mathbb{Z}_m$
denotes the
 topological space $(\mathbb{Z}, \mathcal{F}_m)$. The following result is a simple
generalization of Theorem \ref{main1}; the proof is exactly the
same.

 \begin{theorem} \label{main2}
Let $m\neq0$ and let $f:\mathbb{N}\rightarrow \mathbb{Z}_m$ be a
continuous map, which is unbounded on each congruence class. If
$f(n)$ is prime to $m$ for each  $n\in \mathbb{N}$, then
$\{f(n)\}_{n=0}^{\infty}$ has infinitely many prime divisors.
\end{theorem}

 \section{linear recurrence sequences}

It is well-known that a non-degenerate linear recurrence sequence
of integers of order $k>1$ has infinitely many prime divisors
\cite{ev,la1,lv,po,va,wa1}. In this section, using Theorem
\ref{main2}, we give a simple proof of this result.

First we give some preliminaries about linear recurrence
sequences. We say that a sequence of integers
$\{a_n\}_{n=0}^{\infty}$ is a $linear\ recurrence\ sequence$ of
order $k$ if the following linear recurrence relation of order $k$
is satisfied
\begin{equation}\label{fur4ve}a_{n+k}=r_1a_{n+k-1}+\cdots+r_ka_{n} \  \  \   (n=0,1,2,\cdots),\end{equation} where $r_1,\cdots,r_k\in \mathbb{Z}$ are constant, and $r_k\neq0$.
 We note that a linear recurrence sequence may satisfy linear recurrence relations of different orders.
 For example, the Fibonacci sequence $$a_1=1, a_2=1, a_{3}=2, a_{4}=3,\cdots$$ satisfies both  $$a_{n+2}=a_{n+1}+a_{n},$$ and $$a_{n+3}=3a_{n+2}-a_{n+1}-2a_{n},$$ which are of order 2 and 3, respectively.
 When we say about an order of a linear recurrence sequence and recurrence relation, we always mean the minimal one.

 Let $g(x)=1-\sum_{i=1}^{k}r_ix^i$ be the characteristic polynomial of the recurrence relation (\ref{fur4ve}). If $\{a_n\}_{n=0}^{\infty}$ is a linear recurrence sequence of order $k$,
 satisfying  (\ref{fur4ve}) one can easily check that  $$f(x)=g(x)(\sum_{n=0}^{\infty}a_nx^n)$$ is an integral polynomial of degree less than $k$.
Hence, the generating function of $\{a_n\}_{n=0}^{\infty}$ is  a
rational function
 \begin{equation}\label{fur4}\sum_{n=0}^{\infty}a_nx^n=\frac{f(x)}{g(x)},\end{equation} and $f(x)$ and $g(x)$ are co-prime (otherwise, $\{a_n\}_{n=0}^{\infty}$ would be of order less than $k$, see below for details).
Thus, for a linear recurrence sequence of order $k$, the
generating function extends to a meromorphic function on
$\mathbb{C}$ with $k$ poles (the poles are counted with their
multiplicities). On the other hand, if the generating function of
a sequence of integers $\{a_n\}_{n=0}^{\infty}$  is of the form
(\ref{fur4}), where $g(x)=1-\sum_{i=1}^{k}r_ix^i\in
\mathbb{Z}[x]$, $r_k\neq0$, and $f(x)\in \mathbb{Z}[x]$ is a
nonzero polynomial of degree less than $k$, and $f(x)$ and $g(x)$
are co-prime, then $\{a_n\}_{n=0}^{\infty}$ is a linear recurrence
sequence of order $k$, satisfying (\ref{fur4}).

A recurrence sequence is called $degenerate$ if its characteristic
 polynomial has two distinct roots whose ratio is a root of unity, and $non$-$degenerate$ otherwise.

For $$g(x)=1-\sum_{i=1}^{k}r_ix^i=\prod_{i=1}^{k}(1-\psi_ix)\in
\mathbb{Z}[x],$$ and $b\in \mathbb{N}$, set $$(\phi_b
g)(x)=\prod_{i=1}^{k}(1-\psi_i^b x)\in \mathbb{Z}[x].$$ If the
ratio of two distinct roots of $g(x)$ is not a root of unity, then
the same is true for $\phi_b g$.

\begin{lemma}\label{lemma}
Let $\{a_n\}_{n=0}^{\infty}$ be a non-degenerate linear recurrence
sequence of order $k>1$, and let $$g(x)=1-\sum_{i=1}^{k}r_ix^i$$
be the associated characteristic polynomial. Then for $0\leq c<b$, the  subsequence
$\{a_{c+bn}\}_{n=0}^{\infty}$ is also a non-degenerate linear
recurrence sequence of order $k$, whose characteristic polynomial
is $\phi_b g$.
  \end{lemma}
\begin{proof}
Let the generating function of $\{a_n\}_{n=0}^{\infty}$ be
of the form (\ref{fur4}), with the polynomial $f(x)$ prime to $g(x)$. Then
 the generating function of $\{a_{c+bn}\}_{n=0}^{\infty}$ is
 \begin{equation}\label{fur41}\frac{1}{b}\sum_{i=1}^{b}\zeta_b^{-ci}x^{-c/b}\frac{f(\zeta_b^i x^{1/b})}{g(\zeta_b^i x^{1/b})}=\frac{h(x)}{\phi_b g(x)},\end{equation}
 where $\zeta_b=e^{2\pi i/b}$ and $h(x)\in \mathbb{C}[x]$ is of degree less than $k$ (because the left hand side of (\ref{fur41}) is a rational function vanishing at infinity).
 As the left hand side of (\ref{fur41}) lies in $\mathbb{Z}[[x]]$,  $h(x)$ is in fact an integral polynomial. By counting the poles of both sides of (\ref{fur41}),
 we see that  $\phi_b g(x)$ and $h(x)$ are co-prime; hence the lemma follows.
\end{proof}

For completeness of the paper, we give a short proof of a weaker version of
the celebrated Skolem-Mahler-Lech theorem \cite{va1} (asserting that for every
sequence $\{a_n\}$ as in Corollary \ref{T361} below, $|a_n|\rightarrow\infty$ as $n\rightarrow\infty$).
\begin{corollary} \label{T361}
A non-degenerate linear recurrence sequence
$\{a_n\}_{n=0}^{\infty}$ of order $k>1$ is unbounded on any
congruence class.
\end{corollary}

\begin{proof}
By lemma \ref{lemma}, it suffices to show that
$\{a_n\}_{n=0}^{\infty}$ itself is unbounded. Assume that there
exists a positive integer $m$ such that $|a_n|<m$ for all $n$. As
$\{a_n\}_{n=0}^{\infty}$ is eventually periodic modulo $2m$, there
exists a positive integer$ k$ such that the numbers $a_{nk}$ are
congruent to each other modulo $2m$, for $n$ sufficiently large.
Then $|a_n|<m$ forces $\{a_{nk}\}_{n=0}^{\infty}$ to be eventually
constant; this contradicts Lemma \ref{lemma}.
\end{proof}

A positive integer $m$ will be said to be a $null\ divisor$ of a
sequence of integers $\{a_n\}_{n=0}^{\infty}$ if $a_n$ is
divisible by $m$ for $n$ sufficiently large. For a prime  $p$, the
largest integer $j\in \mathbb{N}$ such that $p^j$ is a null divisor of
$\{a_n\}_{n=0}^{\infty}$ will be called the index of $p$ in
$\{a_n\}_{n=0}^{\infty}$. We need the following result from
\cite{wa} about null divisors.
\begin{lemma}\label{lemma1}
Let $\{a_n\}_{n=0}^{\infty}$ be a non-degenerate linear recurrence
sequence of order $k>1$, and let $$g(x)=1-\sum_{i=1}^{k}r_ix^i$$
be the associated characteristic polynomial. Assume that
$GCD(r_1,\cdots, r_k)=1$. Then for any prime $p$, the index of  $p$
in $\{a_n\}_{n=0}^{\infty}$ is finite.
  \end{lemma}
Now we are in the position to give a  simple proof of the
following theorem \cite{la1}; cf.\cite{ev,lv,po,va,wa1}.
 \begin{theorem} \label{main5}
Let $\{a_n\}_{n=0}^{\infty}$ be a non-degenerate linear recurrence
sequence of order $k>1$, satisfying
$$a_{n+k}=r_1a_{n+k-1}+\cdots+r_ka_{n}, $$
where $r_k\neq0$. Then $\{a_n\}_{n=0}^{\infty}$ has infinitely
many prime divisors.
\end{theorem}

\begin{proof}
Let $$g(x)=1-\sum_{i=1}^{k}r_ix^i=\prod_{i=1}^{k}(1-\psi_ix)\in
\mathbb{Z}[x]$$ be the associated characteristic polynomial. Let $
\mathbb{Q}(\psi_1,\cdots, \psi_k)$ be the splitting field of
$g(x)$ and let $A$ be the integral closure of $\mathbb{Z}$ in $
\mathbb{Q}(\psi_1,\cdots, \psi_k)$. For $\alpha_1,\cdots,
\alpha_n\in A$, let $(\alpha_1,\cdots, \alpha_n)$ denote the ideal
of $A$ generated by $\alpha_1,\cdots, \alpha_n$. The ideal
$(\psi_1,\cdots, \psi_k)$ of $A$ is invariant under the Galois
group. Hence by the ideal theory of  Dedekind rings (cf.
\cite[Theorem 2, p.18]{la} and  \cite[Corollary 2, p.26]{la}),
there exist two positive integers $s,t$ such that
\begin{equation}\label{fur6}(\psi_1^s,\cdots,
\psi_k^s)=(\psi_1,\cdots, \psi_k)^s=(t).\end{equation} Set
\begin{equation}\label{fur61}h(x)=1-\sum_{i=1}^{k}m_ix^i=\prod_{i=1}^{k}(1-\psi_i^sx)\in
\mathbb{Z}[x].\end{equation} We see that $m_i\in \mathbb{Z}$
belongs to $(\psi_1^s,\cdots, \psi_k^s)^{i}=(t^i)$, hence $m_i$ is
divisible by $t^i$.

 \begin{claim}\label{cdrdeem11}
$$GCD(m_1/t^1,\cdots, m_k/t^k)=1.$$
\end{claim}

\begin{proof}
First, by (\ref{fur6}) we have $\psi_i^s/t\in A$ for $1\leq i\leq
k$, and $(\psi_1^s/t,\cdots, \psi_k^s/t) =A$.

Now we show by contradiction  that $(m_1/t^1,\cdots, m_k/t^k)= A$
and this would imply immediately the claim. Assume the contrary
that $(m_1/t^1,\cdots, m_k/t^k)\neq A$. Then there exists a prime
ideal $B$ of $A$ such that $(m_1/t^1,\cdots, m_k/t^k)\in B$. As
$(\psi_1^s/t,\cdots, \psi_k^s/t) =A$, there exists a partition
$I\cup J=\{1,2,\cdots,k\}$ such that $\psi_i^s/t\in B$ if and only
if $i\in I$, and $J$ is nonempty. Set $j=|J|$, and let $f_j(x_{1},
\cdots,x_{k})$ be the $j$-th elementary symmetric polynomial. Then
each product in $f_j(\psi_1^s/t,\cdots, \psi_k^s/t)$ belongs to
$B$ except for $\prod_{i\in J}\psi_i^s/t$, hence
$f_j(\psi_1^s/t,\cdots, \psi_k^s/t)\notin B$. On the other hand,
by (\ref{fur61}), $f_j(\psi_1^s/t,\cdots,
\psi_k^s/t)=(-1)^{j-1}m_j/t^j\in B$.
\end{proof}

By Lemma \ref{lemma}, $\{a_{sn}\}_{n=0}^{\infty}$ is still a
non-degenerate linear recurrence sequence of order $k$, with the
characteristic polynomial
$$h(x/t)=1-\sum_{i=1}^{k}m_ix^i.$$  By induction, we see that $a_{sn}$ is divisible by
$t^n$, for each $n\geq 0$.

Set $b_n=a_{sn}/t^n.$ Then $\{b_{n}\}_{n=0}^{\infty}$ is a
non-degenerate linear recurrence sequence of order $k$, with the
characteristic polynomial
$$h'(x/t)=1-\sum_{i=1}^{k}m_ix^i/t^i.$$ As the prime divisors of $\{b_{n}\}_{n=0}^{\infty}$ are also prime divisors of $\{a_n\}_{n=0}^{\infty}$, it suffices to prove the theorem by  replacing $\{a_n\}_{n=0}^{\infty}$ with
$\{b_n\}_{n=0}^{\infty}$. Hence by Claim \ref{cdrdeem11}, we can assume  that
$GCD(r_1,\cdots, r_k)=1$. This condition still holds for subsequence
$\{a_{ni+l}\}_{n=0}^{\infty}$ where $i>0$.

Let $p_1,\cdots,p_m$ be all the prime divisors of $r_k$. By Lemma
\ref{lemma1}, we can choose a positive integer $l$ which is larger
than  the index of $p_{1}$ in $\{a_n\}_{n=0}^{\infty}$. Then there
exists a positive integer $j$ such that $a_n\equiv a_{n+j} (mod \
p_{1}^l)$ for $n$ sufficiently large. As $p_{1}^l$ is not a null
divisor of $\{a_n\}_{n=0}^{\infty}$, we can choose a subsequence
$\{a_{jn+t}\}_{n=0}^{\infty}$ such that the terms $a_{jn+t}$ are
not divisible by $p_{1}^l $, and are  congruent to each other
modulo $p_{1}^l $. Let $p_{1}^r $ be the highest power of $p_{1} $
dividing $a_{t}$. We can replace $\{a_n\}_{n=0}^{\infty}$ by
$\{a_{jn+t}/p_{1}^r\}_{n=0}^{\infty}$ whose terms are prime to $p_{1}
$.

Continuing in this way, we can finally find a subsequence
$\{a_{cn+d}\}_{n=0}^{\infty}$, $(c>0)$ and a positive integer $e$
such that each term $a_{cn+d}$ is divisible by $e$ and the
quotient $a_{cn+d}/e$ is prime to $p_1,\cdots,p_m$. By Corollary
\ref{T361}, the sequence $\{a_{cn+d}/e\}_{n=0}^{\infty}$ gives a
continuous map from $\mathbb{N}_0$ to $\mathbb{Z}_{r_k}$,
satisfying all the conditions of Theorem \ref{main2}. Now applying
Theorem \ref{main2} we conclude the proof.
\end{proof}

\begin{remark}\label{remariks}
We note that, for each polynomial $f(x)\in \mathbb{Z}[x]$ of
degree $k$, the sequence $\{f(n)\}_{n=0}^{\infty}$ is a
non-degenerate linear recurrence sequence of order $k+1$(see (\ref{for1})). Hence
Theorem \ref{main5} is another generalization of Schur's theorem.
\end{remark}

\end{document}